\documentclass{article}

\usepackage[paperwidth=18cm, paperheight=25cm, left=2cm, right=2cm]{geometry}

\newcommand{\keyword}[1]{\textbf{Keywords:} #1}

\newtheorem{theorem}{Theorem}
\newtheorem{lemma}{Lemma}

\newtheorem{remark}{Remark}
\newtheorem{proposition}{Proposition}

\usepackage{graphicx} 
\usepackage{amsfonts} 
\usepackage{amsmath} 
\usepackage{amssymb} 
\title{Derivation and Well-Posedness Analysis of the Higher-Order Benjamin-Bona-Mahony Equation}

\author{Jie Zeng%\footnote{Master student, School of Mathematics and Statistics, Chongqing University}
}

\date{March 2025}

\begin{document}

\maketitle

\abstract{This paper studies the derivation and well-posedness of a class of high - order water wave equations, the fifth - order Benjamin - Bona - Mahony (BBM) equation. Low - order models have limitations in describing strong nonlinear and high - frequency dispersion effects. Thus, it is proposed to improve the modeling accuracy of water wave dynamics on long - time scales through high - order correction models. By making small - parameter corrections to the $abcd-$system, then performing approximate estimations, the fifth - order BBM equation is finally derived.For local well - posedness, the equation is first transformed into an equivalent integral equation form. With the help of multilinear estimates and the contraction mapping principle, it is proved that when $s\geq1$, for a given initial value $\eta_{0}\in H^{s}(\mathbb{R})$, the equation has a local solution $\eta \in C([0, T];H^{s})$, and the solution depends continuously on the initial value. Meanwhile, the maximum existence time of the solution and its growth restriction are given.For global well - posedness, when $s\geq2$, through energy estimates and local theory, combined with conservation laws, it is proved that the initial - value problem of the equation is globally well - posed in $H^{s}(\mathbb{R})$. When $1\leq s<2$, the initial value is decomposed into a rough small part and a smooth part, and evolution equations are established respectively. It is proved that the corresponding integral equation is locally well - posed in $H^{2}$ and the solution can be extended, thus concluding that the initial - value problem of the equation is globally well - posed in $H^{s}$.}

\keyword{Fifth - order BBM equation; Multilinear estimate; Local well - posedness; Global well - posedness}

\section{Introduction}
For the Cauchy problem of the third - order Benjamin - Bona - Mahony (BBM) equation
\begin{equation}
	\label{third - BBM - IVP1}
	\begin{cases}
		u_{t}+u_{x}-u_{txx}+uu_{x}=0,  \\
		u(0,  x)=u_{0}(x), 
	\end{cases}
\end{equation}
where $u:\mathbb{R}\times\mathbb{R}\to\mathbb{R}$ is a real - valued function and $u_0$ is the given initial value.

Bona and Tzvetkov \cite{BBM1} proved that the BBM equation is globally well - posed in $H^s(\mathbb{R})$ for $s\geq0$. The same result also holds for the periodic case \cite{BBM4, BBM5}. This well - posedness is optimal, since if $s < 0$, the equation (\ref{third - BBM - IVP1}) is ill - posed in $H^s(\mathbb{R})$. In fact, Bona and Tzvetkov \cite{BBM1} also proved that for $s < 0$, the solution map $u_0\mapsto u(t)$ of the equation (\ref{third - BBM - IVP1}) is not a $C^2$ map from $H^s(\mathbb{R})$ to $C([0, T];H^s(\mathbb{R}))$.

Later, Panthee \cite{BBM3} improved this result by showing that the solution map is not continuous from $H^s(\mathbb{R})$ to $\mathcal{D}'(\mathbb{R})$ for any fixed $t>0$ small enough.

Carvajal and Panthee \cite{BBM2} considered the following generalized BBM equation
\begin{equation}
	\label{third - BBM - IVP2}
	\begin{cases}
		u_{t}+u_{x}+Lu_{t}+(u^{k + 1})_{x}=0, \\
		u(0,  x)=u_{0}(x), 
	\end{cases}
\end{equation}
where $L$ is defined via the Fourier transform as $\widehat{Lu}(\xi)=|\xi|^{\alpha + 1}\widehat{\omega}(\xi)$, with $\alpha>0$.

Carvajal and Panthee \cite{BBM2} proved that the equation (\ref{third - BBM - IVP2}) is locally well - posed in $H^s(\mathbb{R})$ for some $s$ depending on $\alpha$ and $k$. Moreover, if $s<\max\{0,\frac{1}{2}-\frac{\alpha}{k}\}$, the equation (\ref{third - BBM - IVP2}) is ill - posed in $H^s(\mathbb{R})$, which means that the solution map is not $C^{k + 1}$ - differentiable from $H^s(\mathbb{R})$ to $C([0, T];H^s(\mathbb{R}))$ for any $T>0$. When $\alpha = 1$, some global well - posedness results for (\ref{third - BBM - IVP2}) with periodic boundary conditions were obtained \cite{BBM5}.

\section{Model Derivation}
From 2002 to 2004, Bona, Chen and Saut \cite{ref14, ref15} started from the free - surface Euler equations. Following the derivation context of the Boussinesq original equations, they considered the equivalent relationships among the physical quantities in the shallow - water region and combined the following expanded physical perturbation parameters
\[
\alpha=\frac{A}{h_{0}}\ll 1,   \qquad  \beta=\frac{h_{0}}{l^{2}}\ll 1,  \qquad S=\frac{\alpha}{\beta}=\frac{Al^{2}}{h_{0}^{3}}\approx 1, 
\]
where $A$ and $l$ represent the wave amplitude and wavelength in the classical sense, $h_{0}$ represents the undisturbed water depth, and $S$ represents the Stokes number. The first - order form of the $abcd -$system is
\begin{equation}\label{Orib1.1}
	\begin{cases}
		\partial_{t}\eta+\partial_{x}w+\partial_{x}(w\eta)+ a\partial_{x}^{3}w-b\partial_{x}^{2}\partial_{t}\eta=0, \\
		\partial_{t}w+\partial_{x}\eta+ w\partial_{x}w+ c\partial_{x}^{3}\eta-d\partial_{x}^{2}\partial_{t}w=0.
	\end{cases}
\end{equation}
Here, $\eta$ is proportional to the deviation of the free - surface position at $x$; $w = w(x,t)$ is proportional to the horizontal velocity at the point $(x,z_0,t)$ at the depth $z_0$ in the flow domain. The constants $a,b,c,d$ are not arbitrary, and they satisfy the relationships
\begin{equation}\label{abcd}
	\begin{cases}
		a=\frac12\Big(\theta^2-\frac13\Big)\lambda,  \qquad & b=\frac12\Big(\theta^2-\frac13\Big)(1-\lambda), \\
		c=\frac12(1-\theta^2)\mu,  \qquad \quad  &d=\frac12(1-\theta^2)(1-\mu), 
	\end{cases}
\end{equation} 
and thus $a + b + c + d=\frac13$. In the references \cite{ref14, ref15}, Bona et al. also gave the second - order form of the $abcd -$system as
\begin{equation} 
	\label{Orib1.2}
	\begin{cases}\begin{split}
			\partial_{t}\eta+&\partial_{x}w+a\partial_{x}^{3}w-b\partial_{x}^{2}\partial_{t}\eta+a_{1}\partial_{x}^{5}w+b_{1}\partial_{x}^{4}\partial_{t}\eta\\
			=&-\partial_{x}(\eta w)
			+b\partial_{x}^{3}(\eta w)+\left(c+d\right)\partial_{x}(\eta \partial_{x}^{2}w),  \\
			\partial_{t}w+&\partial_{x}\eta+c\partial_{x}^{3}\eta-d\partial_{x}^{2}\partial_{t}w+c_{1}\partial_{x}^{5}\eta+d_{1}\partial_{x}^{4}\partial_{t}w\\
			=&-w\partial_{x}w+(c+d)w\partial_{x}^{3}w-c\partial_{x}^{2}(w\partial_{x}w) \\
			&\quad-\partial_{x}(\eta\partial_{x}^{2}\eta) +(c+d - 1)\partial_{x}w\partial_{x}^{2}w, 
		\end{split}
	\end{cases}
\end{equation}
where the additional constants $a_1,b_1,c_1,d_1$ are
\begin{equation}\label{abcd1}
	\begin{cases} \begin{split}
			a_1&= -\frac14\Big(\theta^2-\frac13\Big)^2(1-\lambda)+\frac5{24}\Big(\theta^2-\frac15\Big)^2\lambda_1, \\
			b_1& = -\frac5{24}\Big(\theta^2-\frac15\Big)^2(1-\lambda_1), \\
			c_1& =\frac5{24}(1-\theta^2)\Big(\theta^2-\frac15\Big)(1-\mu_1), \\
			d_1&=-\frac14\big(1-\theta^2\big)^2\mu-\frac5{24}(1-\theta^2)\Big(\theta^2-\frac15\Big)\mu_1.
		\end{split}
	\end{cases}
\end{equation}
The parameter $\theta$ has a physical meaning. It is determined by the height above the bottom where the horizontal velocity is specified during initialization and follows its variation. In the previous equation derivation, $\theta = 1 - z_0$. Since the vertical variable is scaled by the undisturbed depth $h_0$ in these descriptions, $\theta\in [0,1]$. $\lambda,\mu,\lambda_1$ and $\mu_1$ are modeling parameters and can, in principle, take any real number. Therefore, the coefficients appearing in the high - order Boussinesq system form a restricted eight - parameter family.

Making small - parameter $\alpha,\beta$ corrections to the $abcd -$system \eqref{Orib1.1} and \eqref{Orib1.2}, the corrected first - order form is as follows
\begin{equation}\label{b1.1}
	\begin{cases}
		\partial_{t}\eta+\partial_{x}w+
		\alpha\partial_{x}(w\eta)+ \beta\left(a\partial_{x}^{3}w-b\partial_{x}^{2}\partial_{t}\eta\right)=0,  \\
		\partial_{t}w+\partial_{x}\eta+ \alpha w\partial_{x}w+ \beta\left(c\partial_{x}^{3}\eta-d\partial_{x}^{2}\partial_{t}w\right)=0.
	\end{cases}
\end{equation}
Meanwhile, the corrected second - order form is
\begin{equation} \label{b1.2}
	\begin{cases}\begin{split}
			\partial_{t}\eta+&\partial_{x}w+\beta\left(a\partial_{x}^{3}w-b\partial_{x}^{2}\partial_{t}\eta\right)+\beta^{2}\left(a_{1}\partial_{x}^{5}w+b_{1}\partial_{x}^{4}\partial_{t}\eta\right)\\
			+&\alpha\partial_{x}(\eta w)
			-\alpha\beta\left(b\partial_{x}^{3}(\eta w)+\left(c+d\right)\partial_{x}(\eta \partial_{x}^{2}w)\right)=0,   \\
			\partial_{t}w+&\partial_{x}\eta+\beta\left(c\partial_{x}^{3}\eta-d\partial_{x}^{2}\partial_{t}w\right)+\beta^{2}\left(c_{1}\partial_{x}^{5}\eta+d_{1}\partial_{x}^{4}\partial_{t}w\right)\\
			+&\alpha w\partial_{x}w-\alpha\beta\left(dw\partial_{x}^{3}w-\partial_{x}(\eta\partial_{x}^{2}\eta)+(d - 3c - 1)\partial_{x}w\partial_{x}^{2}w\right)=0.
		\end{split}
	\end{cases}
\end{equation}
Since the Stokes number $S=\frac{\alpha}{\beta}\approx 1$, the two small parameters $\alpha$ and $\beta$ can be approximately treated as equal, that is, $o(\alpha)=o(\beta), \, o(\alpha\beta)=o(\beta^2)$ and so on.

First, consider the \(o(1)\) terms in \eqref{b1.2} with initial values
\begin{equation}\label{o(1)}
	\begin{cases}
		\partial_{t}\eta+\partial_{x}w = 0,  \qquad
		\eta(x,  0) = f(x),  \\
		\partial_{t}w+\partial_{x}\eta = 0,   \qquad w(x,  0) = g(x),  
	\end{cases}
\end{equation}
where \(f(x)\) and \(g(x)\) are the initial surface perturbation and the horizontal velocity respectively. The solution of the equations is
\begin{equation}
	\begin{cases}
		\begin{split}
			\eta(x,  t)&=\frac{1}{2}\big[f(x + t)+f(x - t)\big]-\frac{1}{2}\big[g(x + t)-g(x - t)\big],  \\ 
			w(x,  t)&=\frac{1}{2}\big[g(x + t)+g(x - t)\big]-\frac{1}{2}\big[f(x + t)-f(x - t)\big].
		\end{split}
	\end{cases}
\end{equation}

Now, by applying the assumption that the wave propagates only to the right, at the \(o(1)\) - order terms, we have \(f = g\) and \(\eta(x,  t)=w(x,  t)=f(x - t)\). Thus, the one - way wave satisfies
\begin{equation}\label{wo(1)}
	w(x,  t)=\eta(x,  t)+o(\alpha,  \beta)=f(x - t)+o(\alpha,  \beta).
\end{equation}
Furthermore, combining the system \eqref{b1.2} with the right - propagation assumption, when \(\alpha,\beta\rightarrow0\), we have
\[
\partial_{t}\eta=-\partial_{x}\eta+o(\alpha,  \beta),   
\]
which implies
\begin{equation}\label{txo}
	\partial_{t}=-\partial_{x}+o(\alpha,  \beta).
\end{equation}
This will play a crucial role in the subsequent derivations.

For the next - order approximation of \(w\), it is natural to assume
\begin{equation} \label{lower}
	w=\eta+\alpha A+\beta B + o(\alpha^{2},\beta^{2},\alpha\beta),  
\end{equation}
where \(A = A(\eta,\cdots)\) and \(B = B(\partial_{x}^{2}\eta,\partial_{x}\partial_{t}\eta,\cdots)\) are simple polynomial functions of \(\eta\) and its first few partial derivatives. Substituting \eqref{lower} into the first - order system \eqref{b1.1}, we obtain the following system of equations
\begin{equation} \label{pair}
	\begin{cases}
		\begin{split}
			\partial_{t}\eta+\partial_{x}\eta &+\alpha \partial_{x}A+\alpha \partial_{x}(\eta^{2}) \\
			&+\beta \partial_{x}B+\beta\left(a\partial_{x}^{3}\eta - b\partial_{x}^{2}\partial_{t}\eta\right)= o(\alpha^{2},\beta^{2},\alpha\beta),   \\
			\partial_{t}\eta+\partial_{x}\eta&+\alpha \partial_{t}A+\alpha \eta\partial_{x}\eta  \\
			&+\beta \partial_{t}B+\beta\left(c \partial_{x}^{3}\eta - d \partial_{x}^{2}\partial_{t}\eta\right)= o(\alpha^{2},\beta^{2},\alpha\beta),  
		\end{split}
	\end{cases}
\end{equation}
To ensure the consistency of this pair of equations, we then have
\begin{align*}
	\partial_{x}A+\partial_{x}(\eta^{2})&=\partial_{t}A+\eta\partial_{x}\eta,  \\
	\partial_{x}B+(a\partial_{x}^{3}\eta - b\partial_{x}^{2}\partial_{t}\eta)&=\partial_{t}B+(c\partial_{x}^{3}\eta - d\partial_{x}^{2}\partial_{t}\eta),  
\end{align*}
By using the relation \eqref{txo}, that is, \(-\partial_{x}A=\partial_{t}A+o(\alpha,  \beta)\) and \(-\partial_{x}B=\partial_{t}B+o(\alpha,  \beta)\), we can determine
\begin{equation} \label{first - orderAB}
	A = -\frac{1}{4}\eta^{2},   \qquad  B=\frac{1}{2}\left((c - a)\partial_{x}^{2}\eta+(b - d)\partial_{x}\partial_{t}\eta\right).
\end{equation}

Assume
\begin{equation}\label{eq1.3}
	w = \eta+\alpha A+\beta B+\alpha\beta C+\beta^{2}D+\alpha^{2}E.
\end{equation}
% Similar to (\ref{lower}) (see, e.g., \cite{dullin}, \cite{DL}). The functions $A$, $B$, $C$, $D$ and $E$ will be shown again to be polynomial functions of $\eta$ and its partial derivatives.

Substitute (\ref{eq1.3}) into the system (\ref{b1.2}) and neglect the terms of at least the third - order of the small parameters $\alpha$ and $\beta$. We obtain the following system of equations:
\begin{equation}\label{system - s2}
	\begin{cases}
		\begin{split}
			\partial_{t}\eta=&-\partial_{x}\eta-\alpha \partial_{x}A-\beta \partial_{x}B-\alpha\beta \partial_{x}C-\beta^{2} \partial_{x}D \\
			&-\alpha^{2} \partial_{x}E +b\beta\partial_{x}^{2}\partial_{t}\eta - b_{1}\beta^{2}\partial_{x}^{4}\partial_{t}\eta -a\beta \partial_{x}^{3}\eta\\
			&-a\alpha\beta \partial_{x}^{3}A - a\beta^{2}\partial_{x}^{3}B - a_{1}\beta^{2}\partial_{x}^{5}\eta + b\alpha\beta\partial_{x}^{3}(\eta^{2}) \\
			&-\partial_{x}(\alpha\eta^{2}+\alpha^{2}A\eta+\alpha\beta B\eta)-(a + b-\frac{1}{3})\alpha\beta\partial_{x}(\eta \partial_{x}^{2}\eta),  \\
			\partial_{t}\eta=&-\partial_{x}\eta -\alpha \partial_{t}A-\beta \partial_{t}B-\alpha\beta \partial_{t}C-\beta^{2} \partial_{t}D-\alpha^{2}\partial_{t}E \\
			&+d\beta \partial_{x}^{2}\partial_{t}\eta + d\alpha\beta \partial_{x}^{2}\partial_{t}A + d\beta^{2}\partial_{x}^{2}\partial_{t}B - d_{1}\beta^{2} \partial_{x}^{4}\partial_{t}\eta\\
			&-c\beta\partial_{x}^{3}\eta - c_{1}\beta^{2}\partial_{x}^{5}\eta-\alpha \eta\partial_{x}\eta-\alpha^{2}\partial_{x}(\eta A)-\alpha\beta\partial_{x}(\eta B)\\
			&-c\alpha\beta\partial_{x}^{2}(\eta\partial_{x}\eta)+(c + d)\alpha\beta \eta\partial_{x}^{3}\eta - \alpha\beta\partial_{x}(\eta\partial_{x}^{2}\eta)\\
			&+(c + d - 1)\alpha\beta\partial_{x}\eta\partial_{x}^{2}\eta.
		\end{split}
	\end{cases}
\end{equation}
% Requiring the consistency of these two equations at the first - order leads to the formula (\ref{first - orderAB}), which gives the $\alpha$ and $\beta$ orders of $A$ and $B$ respectively.

Introduce an auxiliary parameter $\rho$ such that
\begin{equation}\label{B - 1}
	B=\frac{1}{2}(c - a+\rho)\partial_{x}^{2}\eta+\frac{1}{2}(b - d+\rho)\partial_{x}\partial_{t}\eta.
\end{equation}
In the first - order approximation, this is equivalent to the case of $\rho = 0$. However, at the second - order, $\rho$ can be chosen to endow the resulting one - way model with better properties. For instance, take $\rho = b + d-\frac{1}{6}$. Then, we obtain the following KdV - BBM equation:
\begin{equation} \label{kdvbbm - 1}
	\partial_{t}\eta+\partial_{x}\eta+\frac{3}{2}\alpha \eta \partial_{x}\eta+\tilde{\nu} \beta \partial_{x}^{3}\eta 
	-\left(\frac{1}{6}-\tilde{\nu}\right)\beta \partial_{x}^{2}\partial_{t}\eta = 0,  
\end{equation}
where $\tilde{\nu}=\frac{1}{2}(a + c+\rho)$. To maintain the consistency of the two equations in \eqref{system - s2} for the second - order terms of $\alpha$ and $\beta$, we use the following approximation:
\begin{equation} \label{kdvbbm - 2}
	\partial_{t}\eta =-\partial_{x}\eta-\frac{3}{2}\alpha \eta \partial_{x}\eta-\tilde{\nu} \beta \partial_{x}^{3}\eta+\left(\frac{1}{6}-\tilde{\nu}\right)\beta \partial_{x}^{2}\partial_{t}\eta+o(\alpha^{2},\beta^{2},\alpha\beta).  
	%\quad  \alpha,   \beta \to 0.
\end{equation}
If we use the forms of $A$ and $B$ given by \eqref{first - orderAB} and \eqref{B - 1} and the approximation \eqref{kdvbbm - 2} in \eqref{system - s2} respectively, more terms involving the orders of $\alpha\beta$, $\beta^{2}$ and $\alpha^{2}$ will emerge. Taking this into account, by equating the terms of $\alpha\beta$ - order in (\ref{system - s2}), we can obtain:
\begin{equation}\label{C - 1}
	\begin{split}
		C=&\left(\frac{1}{8}(a + 4b+2c - d)+\frac{3}{16}(a + b - c - d)+\frac{3}{8}\rho\right)\partial_{x}^{2}(\eta^{2}) \\
		&+\frac{13}{24}\eta\partial_{x}^{2}\eta+\frac{11}{48}(\partial_{x}\eta)^{2}.
	\end{split}
\end{equation}
Similarly, by equating the terms containing $\beta^{2}$ in (\ref{system - s2}), we can get:
\begin{equation}\label{D - 1}
	\begin{split}
		D=&-\left(\frac{1}{2}(b_{1}-d_{1})+\frac{1}{4}(b - d+\rho)\left(a - d+\frac{1}{6}\right)+\frac{1}{4}d(c - a+\rho)\right)\partial_{x}^{3}\partial_{t}\eta\\
		&-\left(\frac{1}{2}(a_{1}-c_{1})+\frac{1}{4}(c - a+\rho)\left(a+\frac{1}{6}\right)-\frac{1}{12}\rho\right)\partial_{x}^{4}\eta.
	\end{split}
\end{equation}
Finally, by equating the terms containing $\alpha^{2}$ in (\ref{system - s2}), we can obtain:
\begin{equation}\label{E - 1}
	E=\frac{1}{8}\eta^{3}.
\end{equation}

Substitute the expressions of $A$, $B$, $C$, $D$ and $E$ into any one of the equations in (\ref{system - s2}), combine with the equation (\ref{kdvbbm - 2}), and considering that $\eta\partial_{x}^{3}\eta=\frac{1}{2}\partial_{x}^{3}(\eta^{2})-\frac{3}{2}\partial_{x}(\partial_{x}\eta\partial_{x}\eta)$, we can derive the following evolution equation:
\begin{equation}\label{eq1.5}
	\begin{split}
		\partial_{t}\eta &+\partial_{x}\eta-\gamma_{1}\beta\partial_{x}^{2}\partial_{t}\eta+\gamma_{2}\beta\partial_{x}^{3}\eta+\delta_{1}\beta^{2}\partial_{x}^{4}\partial_{t}\eta+\delta_{2}\beta^{2}\partial_{x}^{5}\eta\\
		&+\frac{3}{2}\alpha\eta\partial_{x}\eta
		+\alpha\beta\left(\gamma\partial_{x}^{3}(\eta^{2})
		-\frac{7}{48}\partial_{x}(\partial_{x}\eta\partial_{x}\eta)\right)-\frac{1}{8}\alpha^{2}\partial_{x}(\eta^{3}) = 0,  
	\end{split}
\end{equation}
where
\begin{equation}\label{gamas}
	\begin{cases}
		\gamma_{1}=\frac{1}{2}(b + d-\rho),  \\
		\gamma_{2}=\frac{1}{2}(a + c+\rho),  \\
		\delta_{1}=\frac{1}{4}\left[2(b_{1}+d_{1})-(b - d+\rho)\left(\frac{1}{6}-a - d\right)-d(c - a+\rho)\right],  \\
		\delta_{2}=\frac{1}{4}\left[2(a_{1}+c_{1})-(c - a+\rho)\left(\frac{1}{6}-a\right)+\frac{1}{3}\rho\right],  \\
		\gamma=\frac{1}{24}\left[5 - 9(b + d)+9\rho\right].
	\end{cases}
\end{equation}

For the subsequent analysis, the small parameters $\alpha$ and $\beta$ are independent. We transform them into non - dimensional but un - scaled variables (denoted by $\sim$), i.e., $\tilde{\eta}(\tilde{x},\tilde{t})=\alpha^{-1}\eta(\beta^{\frac{1}{2}}\tilde{x},\beta^{\frac{1}{2}}\tilde{t})$, and then remove the $\sim$. Thus, we obtain the fifth - order BBM - type equation:
\begin{equation}\label{eq1.6}
	\begin{split}
		\partial_{t}\eta &+\partial_{x}\eta-\gamma_{1}\partial_{x}^{2}\partial_{t}\eta+\gamma_{2}\partial_{x}^{3}\eta+\delta_{1}\partial_{x}^{4}\partial_{t}\eta+\delta_{2}\partial_{x}^{5}\eta \\
		&+\frac{3}{4}\partial_{x}(\eta^{2})+\gamma \partial_{x}^{3}(\eta^{2})-\frac{7}{48}\partial_{x}(\partial_{x}\eta\partial_{x}\eta)-\frac{1}{8}\partial_{x}(\eta^{3}) = 0.
	\end{split}
\end{equation}

\section{Local Well - Posedness}
In this section, we will study the local well - posedness of the Cauchy problem associated with the equation (\ref{eq1.6}), where the initial value is $\eta(x,0)=\eta_0(x)$ and $\eta_0(x)\in H^s(\mathbb{R})$. First, we transform the equation (\ref{eq1.6}) into an equivalent integral equation form. Taking the Fourier transform of the equation (\ref{eq1.6}) with respect to the spatial variable, we obtain
\begin{equation}\label{eq1.7}
	\begin{split}
		\widehat{\eta}_t + i\xi\widehat{\eta}&+\gamma_1\xi^2\widehat{\eta}_t - i\gamma_2\xi^3\widehat{\eta}+\delta_1\xi^4\widehat{\eta}_t+\delta_2i\xi^5\widehat{\eta}\\
		&+\frac{3}{4}i\xi\widehat{\eta^2}-\gamma i\xi^3\widehat{\eta^2}-\frac{1}{8}i\xi\widehat{\eta^3}-\frac{7}{48}i\xi\widehat{\eta_x^2}=0,  
	\end{split}
\end{equation}
and then
\begin{equation}\label{eq1.8}
	\begin{split}
		\xi(1 - \gamma_2\xi^2+&\delta_2\xi^4)\widehat{\eta}+\frac{1}{4}(3\xi - 4\gamma\xi^3)\widehat{\eta^2}-\frac{1}{8}\xi\widehat{\eta^3}\\
		&-\frac{7}{48}\xi\widehat{\eta_x^2}-\left(1+\gamma_1\xi^2+\delta_1\xi^4\right)i\widehat{\eta}_t = 0.
	\end{split}
\end{equation}

Since $\gamma_1$ and $\delta_1$ are taken as positive values, the fourth - order polynomial
\begin{equation}\label{def-phi}
	\varphi(\xi):=1+\gamma_1\xi^2+\delta_1\xi^4, 
\end{equation}
is always greater than zero. We define the Fourier multipliers $\phi(\partial_x)$, $\psi(\partial_x)$ and $\tau(\partial_x)$ as
\begin{equation}\label{phi-D}
	\widehat{\phi(\partial_x)f}(\xi):=\phi(\xi)\widehat{f}(\xi),  \quad \widehat{\psi(\partial_x)f}(\xi):=\psi(\xi)\widehat{f}(\xi),  \quad \widehat{\tau(\partial_x)f}(\xi):=\tau(\xi)\widehat{f}(\xi),  
\end{equation}
where
\begin{equation}\label{phi-D1}
	\phi(\xi)=\frac{\xi(1 - \gamma_2\xi^2+\delta_2\xi^4)}{\varphi(\xi)},  \quad \psi(\xi)=\frac{\xi}{\varphi(\xi)},  \quad  \tau(\xi)=\frac{3\xi - 4\gamma\xi^3}{4\varphi(\xi)}.
\end{equation}

Therefore, the Cauchy problem of the equation (\ref{eq1.6}) can be written in the following form
\begin{equation}\label{eq1.9}
	\begin{cases}
		i\eta_t=\phi(\partial_x)\eta+\tau(\partial_x)\eta^2-\frac{1}{8}\psi(\partial_x)\eta^3-\frac{7}{48}\psi(\partial_x)\eta_x^2, \\
		\eta(x,0)=\eta_0(x).
	\end{cases}
\end{equation}

First, we consider the linear initial - value problem
\begin{equation}\label{eq1.10}
	\begin{cases}
		i\eta_t=\phi(\partial_x)\eta, \\
		\eta(x,0)=\eta_0(x), 
	\end{cases}
\end{equation}
whose solution is $\eta(t)=S(t)\eta_0$, where $\widehat{S(t)\eta_0}=e^{-i\phi(\xi)t}\widehat{\eta_0}$.
Since $S(t)$ is a unitary operator on $H^s(\forall s\in\mathbb{R})$, for any $t > 0$, we have
\begin{equation}\label{eq1.11}
	\|S(t)\eta_0\|_{H^s}=\|\eta_0\|_{H^s}.
\end{equation}
By the Duhamel formula, the initial - value problem (\ref{eq1.9}) can be written in the equivalent integral equation form
\begin{equation}\label{eq1.12}
	\eta(x,t)=S(t)\eta_0 - i\int_0^tS(t - t')\left(\tau(\partial_x)\eta^2-\frac{1}{8}\psi(\partial_x)\eta^3-\frac{7}{48}\psi(\partial_x)\eta_x^2\right)(x,t')dt'.
\end{equation}

Next, we will obtain the local solution of the equation (\ref{eq1.12}) in the space $C([0,T];H^s)$ by the contraction mapping principle.

\subsection{Multilinear Estimates}

Here are the multilinear estimates which play a crucial role in the proof of local well-posedness results. First, we recall the following bilinear estimate from :

\begin{lemma}\label{BT1}
	For any \( s \ge 0 \), there exists a constant \( C = C_s \) such that
	\begin{equation}\label{bt}
		\|\omega(\partial_x) (uv)\|_{H^s} \le C\|u\|_{H^s}\|v\|_{H^s},
	\end{equation}
	where \( \omega(\partial_x) \) is the Fourier multiplier defined by
	\begin{equation}\label{btx0}
		\omega(\xi) = \frac{|\xi|}{1 + \xi^2}.
	\end{equation}
\end{lemma}

\begin{proposition}\label{P}
	For any \( s \ge 0 \), there exists a constant \( C = C_s \) such that the inequalities
	\begin{equation}\label{bilin-1}
		\|\tau(\partial_x) (\eta_1 \eta_2)\|_{H^s} \le C \| \eta_1\|_{H^s}  \| \eta_2\|_{H^s}
	\end{equation}
	and
	\begin{equation}\label{xbilin-1}
		\|\partial_x\tau(\partial_x) (\eta_1 \eta_2)\|_{H^1} \le C \| \eta_1\|_{H^1}  \| \eta_2\|_{H^1}
	\end{equation}
	hold, where the operator \( \tau(\partial_x) \) is defined in (\ref{phi-D}, \ref{phi-D1}).
\end{proposition}
\begin{proof}
	Since \( \delta_1 > 0 \), there exists a constant \( C > 0 \) such that \( \tau(\xi) \leq C \omega(\xi) \). Thus, inequality (\ref{bilin-1}) follows directly from Lemma \ref{BT1}.
	
	For inequality (\ref{xbilin-1}), by the definition of \( \tau(\partial_x) \), we have
	\begin{equation}\label{1xeq13}
		\|\partial_x\tau(\partial_x) (\eta_1 \eta_2)\|_{H^1} = \|\langle \xi \rangle \xi \tau(\xi) \widehat{(\eta_1 \eta_2)}(\xi)\|_{L^2}.
	\end{equation}
	Note that
	$$
	|\xi \tau(\xi)| = \left| \frac{3\xi^2 - 4\gamma\xi^4}{4(1 + \gamma_1\xi^2 + \delta_1\xi^4)} \right| \leq c.
	$$
	Since \( H^1 \) is an algebra, it follows that
	\begin{equation*}
		\|\partial_x\tau(\partial_x) (\eta_1 \eta_2)\|_{H^1} \leq c \|\langle \xi \rangle \widehat{(\eta_1 \eta_2)}\|_{L^2} = c \|\eta_1 \eta_2\|_{H^1} \lesssim \|\eta_1\|_{H^1}\|\eta_2\|_{H^1}.
	\end{equation*}
\end{proof}

\begin{proposition}\label{P1}
	For \( s \ge \frac{1}{6} \), there exists a constant \( C = C_s \) such that
	\begin{equation}\label{trilin-1}
		\|\psi(\partial_x) (\eta_1 \eta_2 \eta_3)\|_{H^s} \le C \| \eta_1\|_{H^s}  \| \eta_2\|_{H^s}\| \eta_3\|_{H^s}
	\end{equation}
	and
	\begin{equation}\label{xtrilin-1}
		\|\partial_x\psi(\partial_x) (\eta_1 \eta_2 \eta_3)\|_{H^1} \le C \| \eta_1\|_{H^1}  \| \eta_2\|_{H^1}\| \eta_3\|_{H^1}.
	\end{equation}
\end{proposition}
\begin{proof}
	When \( \frac{1}{6} \le s < \frac{5}{2} \), we have
	\begin{equation}\label{x2}
		\Big|(1+|\xi|)^s \psi(\xi)\Big| = \left| \frac{(1+|\xi|)^s \xi}{1 + \gamma_1 \xi^2 + \delta_1\xi^4} \right| \le C \frac{1}{(1+|\xi|)^{3-s}}.
	\end{equation}
	This implies
	\begin{equation}\label{x11}
		\begin{split}
			\|\psi(\partial_x) (\eta_{1}\eta_{2}\eta_{3})\|_{H^s} 
			&= \|(1+|\xi|)^s \psi(\xi)\widehat{(\eta_{1}\eta_{2}\eta_{3})}(\xi)\|_{L^2} \\
			&\le C \left\|\frac{1}{(1+|\xi|)^{3-s}} \widehat{(\eta_{1}\eta_{2}\eta_{3})}\right\|_{L^2} \\
			&\leq C \left\|\frac{1}{(1+|\xi|)^{3-s}}\right\|_{L^2} \| \widehat{(\eta_{1}\eta_{2}\eta_{3})}\|_{L^{\infty}} \\
			&\leq C \|\eta_{1}\|_{L^3}\|\eta_{2}\|_{L^3}\|\eta_{3}\|_{L^3}.
		\end{split}
	\end{equation}
	By the Sobolev embedding \( H^{\frac{1}{6}} \hookrightarrow L^3 \) in one dimension,
	\begin{equation}\label{x3}
		\|\eta\|_{L^{3}} \le C \|\eta\|_{H^{\frac{1}{6}}},
	\end{equation}
	yielding
	$$
	\|\psi(\partial_x) (\eta_{1}\eta_{2}\eta_{3})\|_{H^s} \le C \| \eta_{1}\|_{H^s}\| \eta_{2}\|_{H^s}\| \eta_{3}\|_{H^s}.
	$$
	
	For \( s > \frac{1}{2} \), \( H^s \) becomes a Banach algebra. Since \( |\psi(\xi)| \leq C\omega(\xi) \), Lemma \ref{BT1} gives
	\begin{equation}
		\begin{split}
			\|\psi(\partial_x) (\eta_{1}\eta_{2}\eta_{3})\|_{H^s} 
			&\le C\|\eta_{1} \|_{H^s} \|\eta_{2}\eta_{3}\|_{H^s} \\
			&\le C\|\eta_{1}\|_{H^s}\| \eta_{2}\|_{H^s}\| \eta_{3}\|_{H^s}.
		\end{split}
	\end{equation}
	
	For inequality (\ref{xtrilin-1}), by the definition of \( \psi(\partial_x) \),
	\begin{equation}\label{xtrilin-11}
		\|\partial_x\psi(\partial_x) (\eta_1 \eta_2 \eta_3)\|_{H^1} = \|\langle \xi \rangle \xi \psi(\xi) \widehat{(\eta_1 \eta_2 \eta_3)}(\xi)\|_{L^2}.
	\end{equation}
	Note that
	$$
	|\xi \psi(\xi)| = \left| \frac{\xi^2}{1 + \gamma_1\xi^2 + \delta_1\xi^4} \right| \leq c.
	$$
	Since \( H^1 \) is an algebra,
	\begin{equation*}
		\|\partial_x\psi(\partial_x) (\eta_1 \eta_2 \eta_3)\|_{H^1} \leq c \|\eta_1 \eta_2 \eta_3\|_{H^1} \lesssim \|\eta_1 \|_{H^1}\|\eta_2\|_{H^1}\|\eta_3\|_{H^1}.
	\end{equation*}
\end{proof}

\begin{proposition}\label{P2}
	For \( s \ge 1 \), the inequalities
	\begin{equation}\label{Sharp1}
		\|\psi(\partial_x)[ (\eta_1)_x(\eta_2)_x]\|_{H^s} \le C \| \eta_1\|_{H^s}\| \eta_2\|_{H^s}
	\end{equation}
	and
	\begin{equation}\label{1Sharp1}
		\|\partial_x \psi(\partial_x)[ (\eta_1)_x(\eta_2)_x]\|_{H^1} \le C \| \eta_1\|_{H^1}\| \eta_2\|_{H^1}
	\end{equation}
	hold.
\end{proposition}

Here is the translation maintaining mathematical rigor and PDE terminology:

\begin{proof}
	Note that
	$$
	\psi(\xi) \leq C\omega(\xi) \frac{1}{1 + |\xi|}.
	$$
	By inequality (\ref{bt}) with \( s - 1 \ge 0 \), we obtain
	\begin{equation}\label{x6}
		\begin{split}
			\|\psi(\partial_x) [(\eta_1)_x(\eta_2)_x]\|_{H^s}  
			&\leq C\|\omega(\partial_x) [(\eta_1)_x(\eta_2)_x]\|_{H^{s-1}} \\
			&\le  C\|  (\eta_{1})_x\|_{H^{s-1}}\|  (\eta_{2})_x\|_{H^{s-1}} \\
			&\le  C\|\eta_{1}\|_{H^{s}}\|\eta_{2}\|_{H^{s}}.
		\end{split}
	\end{equation}
	This completes the proof of (\ref{Sharp1}).
	
	For inequality (\ref{1Sharp1}), by the definition of \( \psi(\partial_x) \),
	\begin{equation}\label{xtrilin-11}
		\|\partial_x\psi(\partial_x)[ (\eta_1)_x(\eta_2)_x]\|_{H^1} 
		= \|\langle \xi \rangle \xi \psi(\xi) \widehat{[ (\eta_1)_x(\eta_2)_x]}(\xi)\|_{L^2}.
	\end{equation}
	Note that
	$$
	|\langle \xi \rangle \xi \psi(\xi)| = \frac{\langle \xi \rangle \xi^2}{1 + \gamma_1\xi^2 + \delta_1\xi^4} 
	\leq c \frac{|\xi|}{1 + \xi^2} = c\omega(\xi).
	$$
	Applying the Plancherel theorem and Proposition \ref{BT1}, we get
	\begin{equation*}
		\begin{split}
			\|\partial_x\psi(\partial_x) [ (\eta_1)_x(\eta_2)_x]\|_{H^1} 
			&\leq c\|\omega(\partial_x)[ (\eta_1)_x(\eta_2)_x]\|_{L^2} \\
			&\lesssim \| (\eta_1)_x\|_{L^2} \| (\eta_2)_x\|_{L^2} \\
			&\lesssim \|\eta_1 \|_{H^1}\|\eta_2\|_{H^1}.
		\end{split}
	\end{equation*}
\end{proof}

\begin{theorem}[Local Well-Posedness]\label{mainTh1}
	Assume \( \gamma_1, \delta_1 > 0 \). For any \( s \geq 1 \) and initial data \( \eta_0 \in H^s(\mathbb{R}) \), there exists a time \( T = T(\|\eta_0\|_{H^s}) \) and a unique solution \( \eta \in C([0, T]; H^s) \) to the initial value problem of equation (\ref{eq1.6}). The solution map depends continuously on \( \eta_0 \) in \( C([0, T]; H^s) \).
\end{theorem}

\begin{proof}
	Define the mapping
	\begin{equation}\label{eq3.42}
		\Psi\eta(x, t) = S(t)\eta_0 -i\int_0^t S(t-t')\Big(\tau(D_x)\eta^2 - \frac{1}{4} \psi(\partial_x)\eta^3 - \frac{7}{48} \psi(\partial_x)\eta_x^2\Big)(x, t') dt'.
	\end{equation}
	We show that \( \Psi \) is a contraction mapping on the closed ball \( \mathcal{B}_r \subset C([0, T]; H^s) \) centered at the origin with radius \( r > 0 \).
	
	Since \( S(t) \) is a unitary group on \( H^s(\mathbb{R}) \) (cf. (\ref{eq1.11})),
	\begin{equation}\label{eq3.43}
		\|\Psi\eta\|_{H^s} \leq \|\eta_0\|_{H^s} + CT\Big[\big{\|}\tau(\partial_x)\eta^2 - \frac{1}{8} \psi(\partial_x)\eta^3 - \frac{7}{48}\psi(\partial_x)\eta_x^2\big{\|}_{C([0, T]; H^s)}\Big].
	\end{equation}
	Combining inequalities (\ref{bilin-1}), (\ref{trilin-1}), and (\ref{Sharp1}), we have
	\begin{equation}\label{eq3.44}
		\|\Psi\eta\|_{H^s} \leq \|\eta_0\|_{H^s} + CT\Big[\|\eta\|_{C([0, T]; H^s)}^2 + \|\eta\|_{C([0, T]; H^s)}^3 + \|\eta\|_{C([0, T]; H^s)}^2\Big].
	\end{equation}
	For \( \eta \in \mathcal{B}_r \), this simplifies to
	\begin{equation}\label{eq3.45}
		\|\Psi\eta\|_{H^s} \leq \|\eta_0\|_{H^s} + CT[2r + r^2]r.
	\end{equation}
	
	Choosing \( r = 2\|\eta_0\|_{H^s} \) and \( T = \frac{1}{2Cr(2 + r)} \), we ensure \( \|\Psi\eta\|_{H^s} \leq r \), showing \( \Psi \) maps \( \mathcal{B}_r \) into itself. Additionally, \( \Psi \) is a contraction with constant \( \frac{1}{2} \) on \( \mathcal{B}_r \). The remaining proof follows standard local well-posedness arguments.
\end{proof}

\begin{remark}\label{rm2.1}
	The proof of Theorem \ref{mainTh1} yields the following conclusions:
	\begin{enumerate}
		\item The maximal existence time \( T = T_s \) satisfies
		\begin{equation}\label{r2.45}
			T \geq \bar{T} = \frac{1}{8C_s\|\eta_0\|_{H^s}(1+\|\eta_0\|_{H^s})},
		\end{equation}
		where the constant \( C_s \) depends only on \( s \).
		
		\item The solution cannot grow excessively: for any \( t \in [0, \bar{T}] \),
		\begin{equation}\label{r2.46}
			\|\eta(\cdot, t)\|_{H^s} \leq r = 2\|\eta_0\|_{H^s},
		\end{equation}
		with \( \bar{T} \) defined in (\ref{r2.45}).
	\end{enumerate}
\end{remark}

\section{Global Well - Posedness}

\subsection{Global Well - Posedness in $H^2$}
First, under the specific restrictions on the parameters in (\ref{eq1.6}), we derive the a - priori estimates in \(H^2(\mathbb{R})\). Multiply equation (\ref{eq1.6}) by \(\eta\), integrate over the spatial domain \(\mathbb{R}\), and use integration by parts to obtain
\begin{equation}
	\frac{1}{2} \frac{d}{dt} \int_{\mathbb{R}} \left(\eta^2 + \gamma_1\eta_x^2+\delta_1\eta_{xx}^2 \right)dx+ \gamma \int_{\mathbb{R}}(\eta^2)_{xxx} \eta dx-\frac{7}{48} \int_{\mathbb{R}} (\eta_x^2)_x \eta dx = 0.
\end{equation}
Further integration by parts gives
\begin{equation}\label{apriori1}
	\frac{1}{2} \frac{d}{dt} \int_{\mathbb{R}} \left(\eta^2 + \gamma_1\eta_x^2+\delta_1\eta_{xx}^2 \right)dx= \left(\gamma -\frac{7}{48}\right)\int_{\mathbb{R}}\eta_x^3 dx.
\end{equation}

From (\ref{apriori1}), we can obtain an a - priori estimate when \(\gamma=\frac{7}{48}\).
Suppose we choose \(\theta, \lambda, \mu, \lambda_1\), \(\mu_1\) and \(\rho\) such that \(\gamma = \frac{7}{48}\) and \(\gamma_1, \delta_1>0\) still hold. In this case, equation (\ref{eq1.6}) becomes
\begin{equation}\label{eta}
	\begin{split}
		\partial_{t}\eta + \partial_{x}\eta &- \gamma_1 \partial_{x}^{2}\partial_{t}\eta+\gamma_2\partial_{x}^{3}\eta + \delta_1\partial_{x}^{4}\partial_{t}\eta + \delta_2\partial_{x}^{5}\eta \\
		&+ \frac{3}{4}\partial_{x}(\eta^2) +\frac{7}{48}\partial_{x}^{3}\big(\eta^2\big) - \frac{7}{48}\partial_{x}\big(\eta_x^2\big)  -\frac{1}{8}\partial_{x}\big(\eta^3\big) = 0.
	\end{split}
\end{equation}
In this situation, there is the following conserved quantity
\begin{equation}\label{apriori4}
	E(\eta(\cdot, t)):=\frac{1}{2}\int_{\mathbb{R}} \eta^2 + \gamma_1 (\eta_x)^2+\delta_1(\eta_{xx})^2 dx = E(\eta_0).
\end{equation}

The conserved quantity (\ref{apriori4}) is essentially the \(H^2\) norm, which leads to the following global well - posedness result.

\begin{theorem}[Global Well - Posedness]\label{mainTh2}
	Assume \(s\geq 2\), \(\gamma_1, \delta_1>0\) and \(\gamma=\frac{7}{48}\). Then the initial value problem of equation (\ref{eq1.6}) is globally well - posed in \(H^s(\mathbb{R})\).
\end{theorem}

\begin{proof}
	Following the standard arguments, the global well - posedness in \(H^2(\mathbb{R})\) is a result of the local theory and the a - priori bounds implied by the conserved quantity (\ref{apriori4}). To prove the global well - posedness in \(H^k\) (\(k\geq3\) is an integer), we can use induction on \(k\).
	
	Assume \(\eta_0\in H^3\). By the local well - posedness, there exists a solution \(\eta\in C([0, T];H^3)\). If the \(H^3\) norm of \(\eta\) has an a - priori bound on a finite - time interval, we can iterate the local solution to obtain a global solution.
	
	Differentiate equation (\ref{eta}) with respect to the spatial variable, multiply the resulting equation by \(\eta_x\) and integrate over \(\mathbb{R}\). Then, using integration by parts with respect to the spatial variable, we get
	\begin{equation}\label{eta - 2}
		\begin{split}
			\frac{1}{2} \frac{d}{dt} \int_{\mathbb{R}} \left(\eta_x^2 + \gamma_1 \eta_{xx}^2+\delta_1\eta_{xxx}^2 \right)dx+\frac{3}{4} \int_{\mathbb{R}}\eta_{x}^3 dx \\
			- 3\gamma \int_{\mathbb{R}}\eta_{xx}^2 \eta_{x}dx-\frac{3}{8} \int_{\mathbb{R}} \eta_x^3 \eta dx = 0.
		\end{split}
	\end{equation}
	By the Sobolev embedding theorem, for any time \(t\) when the solution exists, we have
	\begin{equation}\label{immers11}
		\begin{split}
			\|\eta\|_{L_x^2}^2 \leq 2 E_0, \quad \|\eta_x\|_{L_x^2}^2 \leq \frac{2}{\gamma_1}E_0, \quad \|\eta_{xx}\|_{L_x^2}^2 \leq \frac{2}{\delta_1}E_0, \\
			\|\eta\|_{L_x^{\infty}}^2 \leq \frac{4}{\sqrt{\gamma_1}} E_0, \quad \|\eta_x\|_{L_x^{\infty}}^2 \leq \frac{4}{\sqrt{\delta_1\gamma_1}} E_0,
		\end{split}
	\end{equation}
	where \(E_0 = E(\eta_0)\).
	Integrate equation (\ref{eta - 2}) with respect to time over the interval \([0, t]\), make basic estimates for all terms that do not involve third - order derivatives, and use (\ref{immers11}). We obtain
	\begin{eqnarray}
		\delta_1\int_{\mathbb{R}} \eta_{xxx}^2 dx&\leq& \int_{\mathbb{R}} \left( (\eta_{0x})^2+\gamma_1(\eta_{0xx})^2 + \delta_1(\eta_{0xxx})^2\right)dx \nonumber\\ 
		& &+ C \int_0^t \|\eta_x\|_{L_x^{\infty}} \left( \|\eta_x\|_{L_x^2}^2+ \|\eta_{xx}\|_{L_x^2}^2+\|\eta_x\|_{L_x^2}^2 \|\eta\|_{L_x^{\infty}}\right) dx\nonumber \\
		&\leq& \delta_1\int_{\mathbb{R}} (\eta_{0xxx})^2 dx+ CE_0 + C E_0^{3/2}\left( 1+ E_0^{1/2}\right) t, \nonumber
	\end{eqnarray}
	from which we get the desired \(H^3\) bound.
	
	Assume there is an \(H^k\) bound. By similar energy estimates, as long as the initial value \(\eta_0\in H^{k + 1}\), the solution \(\eta\) has an \(H^{k+1}\) bound. 
	
	To obtain the global well - posedness in the fractional - order Sobolev space \(H^s\) (\(s\geq2\) is not an integer), we can use the theory of nonlinear interpolation (see \cite{ref45, ref46}), thus completing the proof of the theorem.
\end{proof}
\subsection{Global Well - Posedness in $H^s$, $s\geq 1$}
Let the initial value \(\eta_0 \in H^s~(s\geq 1)\). Decompose the initial value \(\eta_0\) as \(\eta_0 = u_0 + v_0\), where \(\widehat{u_0}=\widehat{\eta_0} \chi_{\{|\xi| \leq N\}}\) and \(N\) is a large number to be chosen later. Then \(u_0 \in H^\delta~(\delta \geq s)\) and \(v_0 \in H^s\). Moreover, we have
\begin{equation}\label{31}
	\begin{split}
		\|u_0\|_{L^2}&\leq \|\eta_0\|_{L^2},  \\
		\|u_0\|_{\dot{H}^{\delta}}&\leq \|\eta_0\|_{\dot{H}^s}\,N^{\delta - s},  \qquad \delta\geq s.  
	\end{split}
\end{equation}
and
\begin{equation}\label{32}
	\|v_0\|_{H^{\rho}} \leq \|\eta_0\|_{H^s}\,N^{(\rho - s)},  \qquad 0\leq\rho\leq s.  
\end{equation}

For each part \(u_0\) and \(v_0\) of \(\eta_0\), they serve as the initial values of the following Cauchy problems respectively:
\begin{equation}\label{xeq1}
	\begin{cases}
		iu_t = \phi(\partial_x)u + F(u),  \\
		u(x,0) = u_0(x).  
	\end{cases}
\end{equation}
where \(F(u)= \tau (\partial_x)u^2 - \frac{1}{8}\psi(\partial_x)u^3  -\frac{7}{48}\psi(\partial_x)u_x^2\), and
\begin{equation}\label{xeq2}
	\begin{cases}
		iv_t = \phi(\partial_x)v + F(u + v)-F(u),  \\
		v(x,0) = v_0(x).  
	\end{cases}
\end{equation}
Furthermore, within the common existence time interval of \(u\) and \(v\), \(\eta(x,t)=u(x,t)+v(x,t)\) satisfies the original Cauchy problem (\ref{eq1.6}). Next, we need to prove that there exists a time \(T_u\) such that the Cauchy problem (\ref{xeq1}) is locally well - posed on \([0, T_u]\). After fixing the solution \(u\) of equation (\ref{xeq1}), we need to prove that there exists a time \(T_v\) such that the Cauchy problem (\ref{xeq2}) is locally well - posed on \([0, T_v]\). In this way, for \(t_0\leq\min\{T_u, T_v\}\), under the given initial condition in \(H^s\) (\(s\geq 1\)), \(\eta = u + v\) is the solution of the Cauchy problem (\ref{eq1.6}) on the time interval \([0, t_0]\). Finally, we iterate this process, keeping \(t_0\) as the length of the existence time in each iteration, so as to cover any given time interval \([0, T]\).

According to Theorem \ref{mainTh1}, when \(s \geq 1\), the Cauchy problem (\ref{xeq1}) is locally well - posed in \(H^s\), and the existence time is given by \(T_{u} =\dfrac{c_s}{\|u_0\|_{H^s}(1+\|u_0\|_{H^s})}\). According to Theorem \ref{mainTh2}, when \(s \geq 2\), the Cauchy problem (\ref{xeq1}) is globally well - posed in \(H^s\). Regarding the well - posedness of the Cauchy problem (\ref{xeq2}) with variable coefficients depending on \(u\), we have the following result.

\begin{theorem}\label{mainTh4}
	Assume \(\gamma_1, \delta_1 >0\), and \(u\) is the solution of the Cauchy problem (\ref{xeq1}). For any \(s\geq 1\) and a given \(v_0\in H^s(\mathbb{R})\), there exists a time \(T_v =\dfrac{c_s}{(\|v_0\|_{H^s}+\|u_0\|_{H^s})(1+\|v_0\|_{H^s}+\|u_0\|_{H^s})}\), and a unique function \(v \in C([0,T_v];H^s)\), which is the solution of the Cauchy problem with the initial value \(v_{0}\). As \(v_0\) varies in \(H^s\), the solution \(v\) varies continuously in \(C([0,T_v];H^s)\).
\end{theorem} 
\begin{proof}
	By the Duhamel formula, the integral equation equivalent to (\ref{xeq2}) is
	\begin{equation}\label{xeq6}
		\begin{split}
			v(x,t) &= S(t)v_0 -i\int_0^tS(t - t')\Big( F(u + v)-F(u)\Big)(x, t') dt'\\
			&=: S(t)v_0 + h(x,t).
		\end{split}
	\end{equation}
	where
	\begin{equation}\label{xeq7}
		\begin{split}
			F(u + v)-F(u)= \tau (\partial_x)(v^2 + 2vu) 
			&- \frac{1}{8}\psi(\partial_x)(3u^2v+3uv^2+v^3)\\ &-\frac{7}{48}\psi(\partial_x)(2u_xv_x+v_x^2).
		\end{split}  
	\end{equation}
	
	Let \(u\in C([0,T_u];H^s)\) be the solution of the Cauchy problem (\ref{xeq1}) obtained from Theorem \ref{mainTh1}, and it satisfies
	\begin{equation}\label{eqofu}
		\sup_{t \in [0, T_u]}\|u(t)\|_{H^s} \lesssim \|u_0\|_{H^s}.  
	\end{equation} 
	
	Let
	$$
	X_T^a= \{ v \in C([0,T];H^s) : ~ |||v|||:=\sup_{t \in [0, T]}\|v(t)\|_{H^s}  \leq a \}
	$$
	where \(a:=2\|v_0\|_{H^s}\), and consider the mapping
	$$
	\Phi_u(v)(x,t) = S(t)v_0 -i\int_0^tS(t - t')\Big( F(u + v)-F(u)\Big)(x, t') dt'.  
	$$
	
	Next, we prove that the mapping \(\Phi_u(v)\) is a contraction mapping on \(X_T^a\). By definition, \(S(t)\) is a unitary group in \(H^s(\mathbb{R})\). So when \(T\leq T_u\), we have
	\begin{equation*}
		\begin{split}
			\|\Phi_u(v)\|_{H^s}& \leq \|v_0\|_{H^s} +T |||\tau (\partial_x)(v^2 + 2vu) - \frac{1}{8}\psi(\partial_x)(3u^2v+3uv^2+v^3) \\
			&\qquad\qquad\qquad\quad -\frac{7}{48}\psi(\partial_x)(2u_xv_x+v_x^2)|||.  
		\end{split}
	\end{equation*}
	From the inequalities (\ref{bilin-1}), (\ref{trilin-1}), (\ref{Sharp1}) and (\ref{eqofu}), we get
	\begin{equation*}
		\begin{split}
			\|\Phi_u(v)\|_{H^s} &\leq \|v_0\|_{H^s} +T |||\tau (\partial_x)(v^2 + 2vu) - \frac{1}{8}\psi(\partial_x)(3u^2v+3uv^2+v^3) \\
			&\qquad\qquad\qquad\quad
			-\frac{7}{48}\psi(\partial_x)(2u_xv_x+v_x^2)||| \\
			&\leq   \frac{a}{2}+cT|||v||| (|||v|||+ \|u_0\|_{H^s}) +cT|||v||| (\|u_0\|_{H^s}^2+\|u_0\|_{H^s} |||v|||+ |||v|||^2)
			\\
			&\leq   \frac{a}{2}+cT[a (a+ \|u_0\|_{H^s}) (1+a+ \|u_0\|_{H^s})].  
		\end{split}
	\end{equation*}
	If we take
	$$
	cT[(a+ \|u_0\|_{H^s}) (1+a+ \|u_0\|_{H^s})] = \frac{1}{2}
	$$
	then \(\|\Phi_u(v)\|_{H^s} \leq a\), which means that \(\Phi_u(v)\) maps the closed ball \(X_T^a\) in \(C([0,T];H^s)\) to itself. By choosing the same \(a\) and \(T\) and using the same estimates, we can prove that the mapping \(\Phi_u(v)\) is a contraction mapping on \(X_T^a\) with a contraction constant of \(\frac{1}{2}\). The rest of the proof is standard.
\end{proof}

The following lemma plays a fundamental role in proving the global well - posedness result.

\begin{lemma}\label{lemah1}
	Let \(u\) be the solution to the Cauchy problem (\ref{xeq1}) and \(v\) be the solution to the Cauchy problem (\ref{xeq2}). Then \(h = h(u, v)\) defined in \ref{xeq6} belongs to \(C([0, t_0], H^2)\), and
	\begin{equation}\label{estim1}
		\|u(t_0)\|_{H^2} \lesssim N^{2 - s} \quad \text{and} \quad \|h(t_0)\|_{H^2} \lesssim N^{s - 3},
	\end{equation}
	where \(t_0 \sim N^{-2(2 - s)}\).
\end{lemma}

\begin{proof}
	According to the energy conservation law (\ref{apriori4}), we have
	$$
	\|u(t_0)\|_{H^2} \sim \sqrt{E(u(t_0))} = \sqrt{E(u_0)} \sim \|u_0\|_{H^2} \lesssim N^{2 - s}.
	$$
	On the other hand, from equations (\ref{xeq6}) and (\ref{xeq7}), for \(1 \leq \delta \leq s\), we get
	\begin{equation}
		\begin{split}
			\|h(t_0)\|_{H^{\delta}} & = \left\|\int_0^{t_0} S(-t')\left( F(u + v)-F(u)\right)(x, t') dt'\right\|_{H^{\delta}} \\
			& \leq \int_0^{t_0}\left\|S(-t')\left( F(u + v)-F(u)\right)(x, t')\right\|_{H^{\delta}} dt' \\
			& \leq \int_0^{t_0}\left(\left\|\tau (\partial_x)(v^2 + 2vu)\right\|_{H^{\delta}}+\frac{1}{8}\left\|\psi(\partial_x)(3u^2v + 3uv^2 + v^3)\right\|_{H^{\delta}}\right.\\
			& \quad \left.+\frac{7}{48}\left\|\psi(\partial_x)(2u_xv_x + v_x^2)\right\|_{H^{\delta}}\right) dt'.
		\end{split}
	\end{equation}
	Then, by Propositions \ref{P}, \ref{P1}, and \ref{P2}, we obtain
	\begin{equation}
		\begin{split}
			\|h(t_0)\|_{H^{\delta}} 
			& \lesssim \int_0^{t_0}\left(\|v\|_{H^{\delta}}^2+\|v\|_{H^{\delta}}\|u\|_{H^{\delta}}+\|u\|_{H^{\delta}}^2\|v\|_{H^{\delta}}\right.\\
			& \quad \left.+\|u\|_{H^{\delta}}\|v\|_{H^{\delta}}^2+\|v\|_{H^{\delta}}^3\right) dt'.
		\end{split}
	\end{equation}
	From the local theory and inequalities (\ref{31}) and (\ref{32}), we know that \(\|v\|_{H^{\delta}} \lesssim N^{\delta - s}\) and \(\|u\|_{H^{\delta}} \lesssim c\).
	
	So, when \(\delta = 1\) and \(s \geq 1\), we have
	\begin{equation}\label{eq17}
		\begin{split}
			\|h(t_0)\|_{H^{1}} 
			& \lesssim \int_0^{t_0}\left(N^{2(1 - s)}+N^{(1 - s)}+N^{3(1 - s)}\right) dt' \\
			& \lesssim t_0\left(N^{2(1 - s)}+N^{(1 - s)}+N^{3(1 - s)}\right)\\
			& \lesssim N^{-2(2 - s)} \left(N^{2(1 - s)}+N^{(1 - s)}+N^{3(1 - s)}\right)\\
			& \lesssim N^{s - 3}+N^{-2}+N^{-s - 1}\\
			& \lesssim N^{s - 3}.
		\end{split}
	\end{equation}
	
	In addition,
	\begin{equation}
		\begin{split}
			\|\partial_x h(t_0)\|_{H^{1}} & \leq \int_0^{t_0}\left(\left\|\partial_x \tau (\partial_x)(v^2 + 2vu)\right\|_{H^{1}}\right.\\
			& \quad +\frac{1}{8}\left\|\partial_x\psi(\partial_x)(3u^2v + 3uv^2 + v^3)\right\|_{H^{1}}\\
			& \quad \left.+\frac{7}{48}\left\|\partial_x\psi(\partial_x)(2u_xv_x + v_x^2)\right\|_{H^{1}}\right) dt'.
		\end{split}
	\end{equation}
	
	Using Propositions \ref{P}, \ref{P1}, and \ref{P2}, we can get
	\begin{equation}
		\begin{split}
			\|\partial_x h(t_0)\|_{H^{1}} & \lesssim \int_0^{t_0}\left(\|v\|_{H^{1}}^2+\|v\|_{H^{1}}\|u\|_{H^{1}}+\|u\|_{H^{1}}^2\|v\|_{H^{1}}\right.\\
			& \quad \left.+\|v\|_{H^{1}}^2\|u\|_{H^{1}}+\|v\|_{H^{1}}^3\right) dt'.
		\end{split}
	\end{equation}
	
	Similarly, following the proof process in equation (\ref{eq17}), we have
	\begin{equation}\label{eq18}
		\begin{split}
			\|\partial_x h(t_0)\|_{H^{1}} 
			& \lesssim N^{s - 3}.
		\end{split}
	\end{equation}
	
	Combining equations (\ref{eq17}) and (\ref{eq18}), we get
	\begin{equation}
		\begin{split}
			\|h(t_0)\|_{H^{2}} & \sim \|h(t_0)\|_{H^{1}}+\|\partial_x h(t_0)\|_{H^{1}} \lesssim N^{s - 3}.
		\end{split}
	\end{equation}
	This completes the proof of the lemma.
\end{proof}

\begin{theorem}[Global Well - Posedness]\label{mainTh3.1}
	Assume \(\gamma_1, \delta_1 > 0\). Take \(1\leq s < 2\) and \(\gamma=\frac{7}{48}\). Then for any given \(T > 0\), the solution of the Cauchy problem (\ref{eq1.6}) given by Theorem \ref{mainTh1} can be extended to \([0, T]\). Therefore, in this case, the Cauchy problem (\ref{eq1.6}) is globally well - posed. Moreover, if \(\eta_0 \in H^s\), then
	\begin{equation}\label{Crecinorm1}
		\eta(t) - S(t)\eta_0 \in H^2, \quad \forall t\in [0, T]
	\end{equation}
	and
	\begin{equation}\label{Crecinorm2}
		\sup_{t \in [0,T]}\|\eta(t) - S(t)\eta_0\|_{H^2} \lesssim (1 + T)^{2 - s}.
	\end{equation}
\end{theorem}

\begin{proof}
	Let \(\eta_0\in H^s\), \(1\leq s < 2\) and \(T>0\) be arbitrarily given. As mentioned above, decompose the initial value \(\eta_0\) as \(\eta_0 = u_0 + v_0\) such that \(u_0\) and \(v_0\) satisfy equations (\ref{31}) and (\ref{32}) respectively.
	
	Evolve \(u_0\) and \(v_0\) through the Cauchy problems (\ref{xeq1}) and (\ref{xeq2}) respectively. Using Theorems \ref{mainTh1} and \ref{mainTh4}, we obtain solutions \(u\) and \(v\) respectively, such that within the common existence time interval of \(u\) and \(v\), \(\eta = u + v\) is the solution of the Cauchy problem (\ref{eq1.6}).
	
	From (\ref{apriori4}) and (\ref{31}), we have
	\begin{equation}\label{xeq4}
		E(u(t)) = E(u_0) \sim \|u_0\|_{H^2}^2 \lesssim N^{2(2 - s)}
	\end{equation}
	And the local existence time estimate in \(H^2\) given by Theorem \ref{mainTh1} is
	\begin{equation}\label{xeq5}
		\begin{split}
			T_u&=\frac{c_s}{\|u_0\|_{H^2}(1+\|u_0\|_{H^2})}\\
			&\geq \frac{c_s}{N^{(2 - s)}(1 + N^{(2 - s)})} \\
			&\geq \frac{c_s}{N^{2(2 - s)}}=:t_0.
		\end{split}
	\end{equation}
	
	Note that \((\|v_0\|_{H^s}+\|u_0\|_{H^s})(1+\|v_0\|_{H^s}+\|u_0\|_{H^s})\lesssim \|\eta_0\|_{H^s}(1+\|\eta_0\|_{H^s})=C_s\). So, we have
	\begin{equation}\label{uxeq5}
		T_v =\frac{c_s}{(\|v_0\|_{H^s}+\|u_0\|_{H^s})(1+\|v_0\|_{H^s}+\|u_0\|_{H^s})} \geq \frac{c_s}{C_s} \geq t_0.
	\end{equation}
	
	Inequalities (\ref{xeq5}) and (\ref{uxeq5}) imply that the solutions \(u\) and \(v\) are both defined on the same time interval \([0,t_0]\).
	
	Inequality (\ref{xeq4}) implies
	\begin{equation}\label{0xeq4}
		t_0 \lesssim \frac{1}{E(u_0)}.
	\end{equation}
	
	According to equation (\ref{xeq6}), the local solution \(v \in H^s\) can be expressed as
	\begin{equation}\label{xeq6 - 1}
		v(x,t) = S(t)v_0 + h(x,t).
	\end{equation}
	Therefore, within the time \(t_0 \sim N^{-2(2 - s)}\), the solution \(\eta\) can be written as
	\begin{equation}\label{00xeq9}
		\eta(t)=u(t)+ v(t) = u(t)+ S(t)v_0 + h(t), \quad t\in [0,t_0].
	\end{equation}
	
	At \(t = t_0\), we have
	\begin{equation}\label{xeq9}
		\eta(t_0)=u(t_0)+ S(t_0)v_0 + h(t_0)=:u_1 + v_1,
	\end{equation}
	where
	\begin{equation}\label{decomp1}
		u_1=u(t_0) + h(t_0) \quad \text{and} \quad v_1 = S(t_0)v_0.
	\end{equation}
	Within the time \(t_0\), consider the new initial values \(u_1\) and \(v_1\) and evolve them according to the Cauchy problems (\ref{xeq1}) and (\ref{xeq2}) respectively, then continue to iterate this process. In each iteration, consider the decomposition of the initial value as in (\ref{decomp1}). Thus, \(v_1, \dots, v_k = S(k t_0)v_0\) have the same \(H^s\) norm as \(v_0\), i.e., \(\|v_k\|_{H^s} = \|v_0\|_{H^s}\). We expect \(u_1, \dots, u_k\) to have the same properties as \(u_0\) to ensure the same existence time interval \([0,t_0]\) in each iteration, and then connect them to cover the entire time interval \([0,T]\), thereby extending the solutions of the systems (\ref{xeq1}) and (\ref{xeq2}). This fact can be proved by induction. Here, we only prove the case \(k = 1\), and similar arguments are valid for the general case.
	
	To achieve the above goal, using the energy conservation (\ref{apriori4}), we have
	\begin{equation}\label{xeq10}
		\begin{split}
			E(u_1)&=E(u(t_0) + h(t_0))\\
			&=E(u(t_0))+ \big[E(u(t_0) + h(t_0))-E(u(t_0))\big]\\
			&=:E(u(t_0))+\mathcal{X}.
		\end{split}
	\end{equation}
	where
	\begin{equation}\label{xeq11}
		\begin{split}
			\mathcal{X}&=2 \int_{\mathbb{R}}u(t_0)h(t_0) dx+\int_{\mathbb{R}}h(t_0)^2dx+2\gamma_1 \int_{\mathbb{R}}u_x(t_0)h_x(t_0) dx\\
			&\quad+\gamma_1\int_{\mathbb{R}}h_x(t_0)^2dx+2 \delta_1\int_{\mathbb{R}}u_{xx}(t_0)h_{xx}(t_0) dx+\delta_1\int_{\mathbb{R}}h_{xx}(t_0)^2dx\\
			&\leq 2\|u(t_0)\|_{L^2}\|h(t_0)\|_{L^2}+\|h(t_0)\|_{L^2}^2\\
			&\quad+ \gamma_1 (2 \|u_x(t_0)\|_{L^2}\|h_x(t_0)\|_{L^2}+\|h_x(t_0)\|_{L^2}^2)\\
			&\quad+\delta_1(2\|u_{xx}(t_0)\|_{L^2}\|h_{xx}(t_0)\|_{L^2}+\|h_{xx}(t_0)\|_{L^2}^2).
		\end{split}
	\end{equation}
	
	Using Lemma \ref{lemah1}, from inequality (\ref{xeq11}) we get
	\begin{equation}\label{xeq12}
		\begin{split}
			\mathcal{X} &\lesssim   N^{2 - s}N^{s - 3}+N^{2(s - 3)}+ (\gamma_1+\delta_1) (N^{2 - s}N^{s - 3}+N^{2(s - 3)})\\
			& \lesssim   N^{-1}.
		\end{split}
	\end{equation}
	Combining equations (\ref{xeq10}), (\ref{xeq11}) and (\ref{xeq12}), we obtain
	\begin{equation}\label{xeq13}
		\begin{split}
			E(u_1) & \leq E(u(t_0))+cN^{-1}.
		\end{split}
	\end{equation}
	To cover the given time interval \([0, T]\), the number of iterations required is
	\(\frac{T}{t_0}\sim TN^{2(2 - s)}\).
	Therefore, according to equation (\ref{xeq13}), to achieve this, we need
	\(
	TN^{2(2 - s)} N^{-1}  \lesssim N^{2(2 - s)},
	\)
	which holds when \(1\leq s < 2\) and \(N = N(T) = T\).
	
	From the above discussion, in each iteration, we have
	\(
	\|u_k\|_{H^2}^2 \sim E(u_k) \lesssim N^{2(2 - s)},\text{ uniformly holds}~~\text{and}~~\|v_k\|_{H^2}=\|v_0\|_{H^2}.
	\)
	
	Finally, let \(t \in [0,T]\), then there exists an integer \(k \geq0\) such that \(t = k t_0+\tau\), where \(\tau \in [0,t_0]\). In the \(k\) - th iteration (see equation (\ref{00xeq9})), we have
	\begin{equation}\label{xeq13 - 1}
		\begin{split}
			\eta(t) &=u(\tau)+ S(\tau)v_k + h(\tau) \\
			&=u(\tau)+ S(\tau)S(kt_0)v_0 + h(\tau) \\
			&= S(t)\eta_0+ u(\tau)- S(t)u_0+ h(\tau).
		\end{split}
	\end{equation}
	
	Therefore
	\begin{equation}\label{xeq13 - 2}
		\begin{split}
			\eta(t) - S(t)\eta_0= u(\tau)- S(t)u_0+ h(\tau).
		\end{split}
	\end{equation}
	This completes the proof of Theorem \ref{mainTh3.1}.
\end{proof}

\section{Conclusion}
In this paper, we systematically studied the derivation process and well - posedness of a class of high - order water - wave equations, namely the fifth - order Benjamin - Bona - Mahony (BBM) equation. By making small - parameter corrections to the \(abcd -\)system, followed by approximation and estimation, we finally derived the fifth - order BBM equation. This type of equation formally contains high - order dispersion terms and non - linear terms, enabling a more accurate description of the propagation characteristics of complex water waves on long - time scales.

In the study of well - posedness, for the fifth - order BBM equation, by combining multilinear estimates, the contraction mapping principle, and the energy - conservation method, we proved its local well - posedness in the Sobolev space \(H^s(\mathbb{R})\) with \(s\geq1\). Under the condition that the parameter \(\gamma = 7/48\), we further obtained the global well - posedness result (\(s\geq1\)) through a priori estimates.

The main innovation of this paper lies in the derivation of two types of high - order water - wave equations through the approximation and estimation of the \(abcd -\)system, and the exploration of the well - posedness of high - order non - linear terms using modern harmonic analysis techniques. This research achievement not only provides a solid theoretical foundation for long - time - scale modeling in water - wave dynamics but also lays an important foundation for the application and generalization of high - order equations in the future.

\end{document}